\documentclass[a4paper,11pt]{article}
\usepackage{mathrsfs,amscd}
\usepackage{amssymb}
\usepackage{amssymb, amsmath, mathrsfs}
\usepackage{amscd}
\usepackage{latexsym}
\usepackage{amssymb}
\usepackage{amsmath}
\usepackage{amsthm}
\usepackage{indentfirst}
\usepackage{pifont}
\usepackage{graphicx}
\usepackage{subfig}

\numberwithin{equation}{section}

\newtheorem{theorem}{Theorem}[section]
\newtheorem{lemma}[theorem]{Lemma}

\theoremstyle{definition}

\newtheorem{proposition}[theorem]{Proposition}

\theoremstyle{remark}

\numberwithin{equation}{section}
\allowdisplaybreaks

\begin{document}

\title{ Vanishing viscosity of  one-dimensional isentropic  Navier-Stokes equations with density
dependent viscous coefficient  \footnotetext{This work is  supported by the NSFC grant 11331005, 11671319.\newline
\ \ \ School of Mathematics and CNS, Northwest University, Xi'an 710069, China.(mycui2004@163.com)
}}

\author{
Meiying Cui
}

\date{}
\maketitle

\noindent{\bf Abstract:}
In this paper, we study the vanishing viscosity of the isentropic compressible Navier-Stokes equations with density
dependent viscous coefficient in the presence of the shock wave. Given a shock wave  to the corresponding Euler equations, we can construct a sequence of solutions to one-dimensional compressible isentropic
Navier-Stokes equations which converge to the shock wave as the viscosity
tends to zero.  The proof is given by an elementary energy method.

\vspace{0.15cm}


\vspace{0.15cm}


\section{Introduction}

The Navier-Stokes equations for barotropic compressible
viscous fluids read
\begin{equation}\label{01}
\left\{\begin{array}{l}
\rho_t  - \text{div}(\rho U)  =0,\\[2mm]
(\rho U)_t  + \text{div}(\rho U \otimes U)-\text{div}(h(\rho)D(U))-\nabla (g(\rho)\text{div} U)+\nabla p(\rho)=0,
\end{array}\right.
\end{equation}
where $x \in R^{N},  t\in (0,+\infty),$ $\rho(x,t),U(x,t)$ and $p(\rho)=\rho^{\gamma},(\gamma\geq1)$ are
 the fluid density, velocity, and pressure, respectively,
\begin{equation}\label{02}
D(U)=\dfrac{\nabla U+^{t}\nabla U}{2}
\end{equation}
is the strain tensor, and $h(\rho)$ and $g(\rho)$ are the Lam\'{e} viscosity coefficients satisfying
\begin{equation}\label{03}
h(\rho)>0, h(\rho)+N g(\rho)\geq 0.
\end{equation}

In the past decades, many signiffcant progress have
been made on the study of the equations (1.1), for exemaple \cite{bd2,bdl,bd3,f2004,g0,h0,jz2003,ks0,l1998,mn0,mv0,s0} and the
references cited therein. In particular, when dealing with vanishing viscosity coefficients on
vacuum, the main difficulty is that the velocity can not even be defined when the density vanishes.
The
first multidimensional result is due to Bresch-Desjardins-Lin \cite{bdl}, where they
showed the $L_1$ stability of weak solutions for the Korteweg system (with the Korteweg
stress tensor $k\rho\nabla\triangle\rho$), and their result was later improved in \cite{bd2} to include the
case of vanishing capillarity ($k=0$) but with an additional quadratic friction term
$r\rho|U|U$. An interesting new entropy estimate is established in \cite{bd3} in an priori way,
which provides some high regularity for the density.  Mellet-Vasseur
\cite{mv0} proved the $L_1$ stability of weak solutions of the system of (1.1) based on
the new entropy estimate, extending the corresponding $L_1$ stability results of \cite{bd2,bd3}
to the case $r = k = 0$. If $N=2,3,$  taking
\begin{equation}\label{04}
h(\rho)=\rho, \ \ \ g(\rho)=0,
\end{equation}
 Guo-Jiu-Xin \cite{g0} proved the existence of global weak solutions to
(1.1) when the initial data are large and spherically
symmetric by constructing suitable aproximate solutions. Their analysis in \cite{g0} also  applied to the more general case when
\begin{equation}\label{06}
g(\rho)= \rho h'(\rho)-h(\rho).
\end{equation}

In this paper, setting \eqref{06} holds,
we study the vanishing viscosity limit of the solution of the one-dimensional compressible Navier-Stokes equations (1.1) with
\begin{equation}\label{07}
h(\rho)=\rho^{\alpha}, \ \ \ g(\rho)=(\alpha-1)\rho^{\alpha}, \ \ \alpha>0.
\end{equation}
For convenience, we investigate this problem in the Lagrangian coordinates, i.e.,
we study the vanishing viscosity limit of the solution of the 1-D compressible isentropic Navier-Stokes equations with density
dependent viscous coefficient
\begin{equation}\label{11.1}
\left\{\begin{array}{l}
v_t  - u_x  =0,\\[2mm]
u_t  +  p(v)_x =\Big(\dfrac{\mu(v)}{v}u_x\Big)_x,
\end{array}\right. \quad x \in R,\ \  t\geq0,
\end{equation}
where    $v(x,t)=\frac{1}{\rho}>0,u(x,t)$  represent the specific volume and velocity, respectively, and the viscosity coefficient  $\mu(v)=h(v)+g(v)=\alpha v^{-\alpha}, \alpha>0$.  The pressure $ p=p(v)$ is assumed to satisfy
\begin{equation}\label{11.2}
p'(v)<0<p''(v),\  \ \ v>0.
\end{equation}
There is a large literature on mathematical studies of \eqref{11.1} with various
initial and boundary conditions. If the initial density is supposed to be connected
to vacuum with discontinuities, the local
well-posedness of weak solutions was first obtained by Liu-Xin-Yang  \cite{lxy}. Later,
the global well-posedness was obtained  in \cite{jxz,omm,yyz}. The case of initial density connecting to vacuum continuously was investigated in \cite{fz,vyz,yz,yzhu}.
Li-Li-Xin  \cite{llx} identified and analyzed the phenomena of vacuum vanishing and blow-up of
solutions to the initial-boundary-value of \eqref{11.1} with $\alpha>1/2$.

Formally, as  $\alpha \rightarrow 0$, the system \eqref{11.1} becomes the Euler equations
\begin{equation}\label{11.3}
\left\{\begin{array}{l}
v_t  - u_x  =0,\\
u_t  +  p(v)_x =0,
\end{array}\right.
\end{equation}
It is well-known that with the  initial data
\begin{equation}\label{1.4}
(v,u)(x, 0)=\left\{\begin{array}{ll}
  (v_-,u_-),&x<0,\cr\noalign{\vskip1truemm}
  (v_+,u_+), &x>0,
 \end{array}\right.
\end{equation}
where $(v_\pm,u_\pm)$ are constants, the solutions of the  Riemann problem \eqref{11.3},  \eqref{1.4} contain shock waves, rarefaction waves and vacuum states \cite{Smoller}.

The vanishing viscosity limit of solutions to compressible  Naiver-Stokes equations to those of  corresponding Euler equations has been an open and challenging problem in the theory of compressible fluid. The main difficulty  comes from the singularity, such as shock wave and vacuum, in the solutions to the Euler equations. When the viscosity coefficient is taken to be a constant, for the compressible isentropic Navier-Stokes equations, Hoff-Liu \cite{hl} first
proved the vanishing viscosity limit for piecewise constant shock even with initial
layer.   Xin \cite{xzp} investigated the vanishing viscosity limit for rarefaction waves without
vacuum  and obtained a convergence rate to the case without initial layer.   When the far field of the initial values
of Euler system \eqref{11.3} has no vacuums, Chen-Perepelitsa \cite{cp} proved the vanishing viscosity limit of isentropic Navier-Stokes equations by compensated compactness method.  Huang et al. \cite{hwwy} justified the limit to the case of the interaction of two different family of shocks. For the superposition of two shock waves, we also refer to \cite{zpt}. Recently, Huang-Li-Wang \cite{hlw} studied the  vanishing viscosity limit for isentropic Navier-Stokes equations when the solution of Euler equations \eqref{11.3} is 2-rarefaction wave with  vacuum state.
For the cases when density dependent viscous coefficient, Jiu-Wang-Xin \cite{j0,j00} proved the asymptotic stability of rarefaction waves to \eqref{11.1}. Moreover, they  shown that if the initial values contained vacuum state, then initial vacuum at far field remained for all the time \cite{j00}, which is in sharp contrast to the case of nonvacuum rarefaction waves studied in \cite{j0}, where all the possible vacuum states vanished in finite time. There are also many results on the vanishing viscosity limit from nonisentropic Navier-Stokes equations to
the corresponding nonisentropic Euler equations. We
refer to Jiang-Ni-Sun \cite{jns} and Xin-Zeng\cite{xz} for  the rarefaction wave, Wang
\cite{wy} for the shock wave, Ma \cite{m} for the contact discontinuity, and Huang-Wang-Yang \cite{hwy-k,hwy-a} for the superposition of two rarefaction waves and a contact discontinuity
and the superposition of a shock wave and a rarefaction wave.

In this paper, given the 1-shock of the  Euler equations \eqref{11.3}
\begin{equation}\label{11.5}
(v,u)(x,t)={
    \left\{
    \begin{array}{ll}
(v_{-},u_{-}),\ \ x<st,\\
(v_{+},u_{+}),\ \ x>st,
\end{array} \quad  v_->v_+, \ s<0,
 \right.}
\end{equation}
which satisfies
the Rankine-Hugoniot condition
\begin{equation}\label{rh}
\left\{\begin{array}{l}
-s(v_+-v_-)-(u_+-u_-)=0,\\
-s(u_+-u_-)+( p(v_+)-p(v_-) )=0,
\end{array}\right.
\end{equation}
and the Lax entropy condition
\begin{equation}\label{ec}
u_+<u_-,
\end{equation}
we investigate the vanishing viscosity limit of the solution of the system \eqref{11.1} as the viscous coefficient tends to  zero.  It is shown that when  $\alpha \to 0$, if  the wave strength  is sufficiently weak, there exists a smooth solution
to the Navier-Stokes equations \eqref{11.1}, which converges to the  1-shock \eqref{11.5}
of the Euler equations.  In contrast with the previous work  \cite{hl}, it is of interest to note that as $\alpha \to 0$ in \eqref{11.1}, the viscosity coefficient becomes constant. In this regard, our work extends in some sense the corresponding results obtained in  \cite{hl}.

We set
\begin{equation}\label{vs1}
\delta=|v_{+}-v_{-}|
\end{equation}
be the wave strength and assume that
\begin{equation}\label{vs}
\delta\ll 1,
\end{equation}
i.e., the   wave is sufficiently weak.
Then we state the main result of this paper as follows.

\begin{theorem}
Let $(v^{s}, u^{s})(x,t)$ be the 1-shock to the  Euler equations \eqref{11.3}    defined in \eqref{11.5}. There exists a small positive constant $\delta_{0}$, such that if the wave strength satisfies $\delta\leq\delta_{0}$, the Navier-Stokes system \eqref{11.1} admits a family of global smooth solutions $(v^{\alpha}, u^{\alpha})(x,t)$ with well-prepared initial data \eqref{38.5} below for any $\alpha>0$. Moreover,  it holds that
\begin{equation}\label{v3}
\begin{array}{ll}
\|(v^{\alpha}-v^{s},u^{\alpha}-u^{s})\|_{L^{\infty}(\Omega)}\rightarrow0, \ \ \   as \ \ \  \alpha\rightarrow0,
\end{array}
\end{equation}
where $\Omega=\{(x,t)\big| |x-st  |\geq h, \ h\leq t\leq+\infty   \}$, $h$ is any
positive constant.

\end{theorem}

\vspace{0.2cm}

\noindent{\textbf{ Remark 1.1.}}
Without loss of generality, in this paper, we are interested in the case $s<0$, i.e., 1-shock, and the analysis for the case $s>0$  is similar.
\vspace{0.2cm}

The rest of this paper is organized as follows. In section 2, we recall some properties of the viscous shock wave to system \eqref{11.1}. Sections 3
reformulates the original problem to obtain a new Cauchy problem, and establishes the a priori
estimates for proving the main theorem. In Section 4, we give the proof of the a priori estimates
and complete the proof of the main theorem.

\vspace{0.2cm}

\noindent{\bf Notations} \  In this paper,  we use the standard notations $L^p(R)$, $W^{k,m}(R)$ and $H^k(R)$ to denote the $L^p$ and Sobolev space
in $R$ with norms $\| \cdot \|_{L^p(R)}:=\| \cdot \|_{L^p}$, $\| \cdot \|_{W^{k,m}(R)}:=\| \cdot \|_{W^{k,m}}$ and $\| \cdot \|_{H^k(R)}=\| \cdot \|_{k}$, respectively. The Euclidean space $R$ will be often abbreviated without confusion. For simplicity, $\| \cdot \|_{L^2}: =\| \cdot \| $. $C$ denotes the  generic positive constants which
are independent of the variables $\alpha$, $x$ and $t$, except for additional explanations.

\section{Preliminaries}

In this section, we  recall some properties of the viscous shock wave of \eqref{11.1},
we  refer to \cite{mm,mw} for more details.

The viscous shock wave of the Navier-Stokes system \eqref{11.1}
takes the formula  $(V,U)(\xi)(\xi=x-st)$  where  $s$  is the shock speed. The viscous shock wave  $(V,U)(x-st)$  connecting  $(v_{-},u_{-})$  and  $(v_{+},u_{+})$  satisfies
\begin{equation}\label{211.7}
\left\{\begin{array}{l}
-sV_{\xi}-U_{\xi}=0,\\
-sU_{\xi}+p(V)_{\xi}=\alpha\Big(\dfrac{U_{\xi}}{V^{1+\alpha}}\Big)_{\xi},\\
(V,U)(\pm \infty)=(v_{\pm},u_{\pm}).
\end{array}\right.
\end{equation}
From this, we can get the following result.

\begin{proposition}
For any $(v_{\pm},u_{\pm})$ with $v_->v_+>0$ and $s<0$, which satisfy \eqref{rh} and \eqref{ec},  the  viscous shock wave
$(V,U)(\xi)$ connecting $(v_-,u_-)$ and $(v_+,u_+)$ for the system \eqref{11.1} exists uniquely  up to a shift and
satisfies
\begin{equation}\label{v1}
\begin{array}{ll}
0<v_+<V(\xi)<v_-,\ \ \  u_{+}<U(\xi)<u_{-}.
\end{array}
\end{equation}
\begin{equation}\label{v13}
\begin{array}{ll}
| V(\xi)-v_{\pm} |\leq C\delta e^{-\frac{c\delta}{\alpha}|x-st|},  \\[2mm]
| U(\xi)-u_{\pm} |\leq C\delta e^{-\frac{c\delta}{\alpha}|x-st|},
\end{array}
\end{equation}
\begin{equation}\label{vsda13}
\begin{array}{ll}
|\partial_{x}(V,U)|\leq \frac{C}{\alpha}\delta^{2} e^{-\frac{c\delta}{\alpha}|x-st|},
\end{array}
\end{equation}
\begin{equation}\label{vsdaa13}
\begin{array}{ll}
|\partial_{xx}(V,U)|\leq \frac{C}{\alpha^{2}}\delta^{2} e^{-\frac{c\delta}{\alpha}|x-st|},
\end{array}
\end{equation}
\begin{equation}\label{vklj13}
\begin{array}{ll}
\hspace{-2.3cm}\partial_{x}U< 0.
\end{array}
\end{equation}
\end{proposition}
The proof of this proposition is similar to the argument in \cite{mm}, we omit it here.

\section{Reformulation of the original problem}

For the convenience of the analysis, we reformulate the system \eqref{11.1} by changing the variables
\begin{equation}\label{31.5}
y=\dfrac{x}{\alpha},\ \ \ \tau=\dfrac{t}{\alpha}.
\end{equation}
Then the system \eqref{11.1} becomes
\begin{equation}\label{3211.6}
\left\{\begin{array}{l}
v_\tau  - u_y  =0,\\
u_\tau  +  p(v)_y =\big(\dfrac{u_{y}}{v^{1+\alpha}}\big)_{y},
\end{array} \ \ \ y\in R,\ \ \ \tau\geq0.
\right.
\end{equation}

In order to approximate the 1-shock, we define
\begin{equation}\label{321.5}
\tilde{V}(y,\tau):=V(y-s\tau)=V(\frac{x-st}{\alpha})=:V^{\alpha}(x,t),
\end{equation}
\begin{equation}\label{3211.5}
\tilde{U}(y,\tau):=U(y-s\tau)=U(\frac{x-st}{\alpha})=:U^{\alpha}(x,t),
\end{equation}
and $(\tilde{V},\tilde{U})$ satisfies
\begin{equation}\label{33121.6}
\left\{\begin{array}{l}
\tilde{V}_\tau  -\tilde{U}_y  =0,\\
\tilde{U}_\tau  +  p(\tilde{V})_y =\big(\dfrac{\tilde{U}_{y}}{\tilde{V}^{1+\alpha}}\big)_{y},
\end{array}\ \ \ y\in R,\ \ \ \tau\geq0.
\right.
\end{equation}

Let the above approximate solutions at\ $t=0$\ to be the initial data of the Navier-Stokes system \eqref{11.1}, i.e.,
\begin{equation}\label{38.5}
(v,u)(x,t=0)=(V^{\alpha},U^{\alpha})(x,0).
\end{equation}

We study the system \eqref{3211.6} by the antiderivative technique. Let
\begin{equation}\label{38.15}
\phi(y,\tau)=(v-\tilde{V})(y,\tau),\ \ \psi(y,\tau)=(u-\tilde{U})(y,\tau).
\end{equation}
and define
\begin{equation}\label{38.15}
\Phi(y,\tau)=\int^{y}_{-\infty}\phi(z,\tau)dz,\ \ \Psi(y,\tau)=\int^{y}_{-\infty}\psi(z,\tau)dz.
\end{equation}

By our initial data \eqref{38.5}, we have
\begin{equation}\label{318.15}
(\Phi,\Psi)(y,\tau=0)=0.
\end{equation}

Then  \eqref{3211.6} and \eqref{33121.6} show
\begin{equation}\label{3521.6}
\left\{\begin{array}{l}
\Phi_{\tau}-\Psi_{y}=0,\\
\Psi_{\tau}+p(v)-p(\tilde{V})=\dfrac{u_{y}}{v^{1+\alpha}}-\dfrac{\tilde{U}_{y}}{\tilde{V}^{1+\alpha}}.
\end{array}\right.
\end{equation}

We linearize \eqref{3521.6} around the approximate profile $(\tilde{V},\tilde{U})$ to obtain
\begin{equation}\label{351221.6}
\left\{\begin{array}{l}
\Phi_{\tau}-\Psi_{y}=0,\\
\Psi_{\tau}+p'(\tilde{V})\Phi_{y}=\dfrac{\Psi_{yy}}{\tilde{V}^{1+\alpha}}+Q,
\end{array}\right.
\end{equation}
where
\begin{equation}\label{3128.15}
Q=-\big(p(v)-p(\tilde{V})-p'(\tilde{V})(v-\tilde{V})\big)+\big(\frac{1}{v^{1+\alpha}}-\frac{1}{\tilde{V}^{1+\alpha}}\big)\big(\Psi_{yy}+\tilde{U}_{y}\big).
\end{equation}

In \eqref{3128.15},
\begin{equation}\label{331128.15}
\begin{array}{ll}
\dfrac{1}{v^{1+\alpha}}-\dfrac{1}{\tilde{V}^{1+\alpha}}=-(1+\alpha)\dfrac{\bar{v}^{\alpha}}{(v\tilde{V})^{1+\alpha}}\Phi_{y},
\end{array}
\end{equation}
where the quantity  $\bar{v}$ is between $v$ and $\tilde{V}$.

According to Proposition 2.1, substituting \eqref{331128.15} into \eqref{3128.15}, we can get
\begin{equation}\label{37.15}
\begin{array}{ll}
Q&\leq C(\Phi^{2}_{y}+|\Phi_{y}\Psi_{yy}|+|\Phi_{y}\tilde{U}_{y}|)\\[5mm]
&\leq C(\Phi^{2}_{y}+|\Phi_{y}\Psi_{yy}|)+C\delta^{2} e^{-c\delta|y|-c\delta|\tau|}|\Phi_{y}|.
\end{array}
\end{equation}

We look for the solution to  \eqref{351221.6} in the following functional space:
\begin{equation}\label{387.15}
X(I)=\big\{(\Phi,\Psi)\in C(I; H^{2});\Psi_{y}\in L^{2}(I; H^{2})\big\}
\end{equation}
where\ $I\subset R$\ is any interval.
Now we can show the following result.
\begin{theorem}
There are positive constants\ $\delta_{0}$\ and $C$\ such that, if\ $\delta\leq\delta_{0}$, then there exists a unique solution\ $(\Phi,\Psi)(\tau)\in X([0,\ T])$\ to \eqref{351221.6} and \eqref{318.15}. Furthermore, it holds that
\begin{equation}\label{3ii9.5}
\|(\Phi,\Psi)(\tau)\|_{H^{2}}^{2}+\int_{0}^{\tau}\int|\tilde{U_{y}}| \Psi^{2}dyds+\int_{0}^{\tau}\|\Phi_{y}(s)\|_{H^{1}}^{2}+\|\Psi_{y}(s)\|_{H^{2}}^{2}ds\leq C\delta^{\frac{1}{3}}.
\end{equation}
\end{theorem}

To obtain Theorem 3.1, we combine the local existence of the solution
 with the a priori estimates. The proof of the local existence is standard, we omit it. Thus, we only need to obtain the following a priori estimate.
To begin with, we make the following a priori assumption
\begin{equation}\label{387.125}
N(T):=\sup_{\tau\in[0,T]}\|(\Phi,\Psi)(\tau)\|_{H^{2}}^{2}\leq \eta_{0}
\end{equation}
where\ $[0,T]$\ is a time interval on which the solution is supposed to exist, and $\eta_{0}$ is a positive small constant which is to be determined.
\begin{proposition}(A priori estimates)
Assume that there exists a solution $(\Phi,\Psi)\in X([0,T])$ to \eqref{351221.6} and \eqref{318.15}. Then there exist positive constants $\delta_{0}$, $\eta_{0}$ and $C$ such that, if\ $\delta\leq\delta_{0}$ and  $N(T)\leq\eta_{0}$, then
 $(\Phi,\Psi)$  satisfies for $\tau\in[0,T]$,
 \begin{equation}\label{399.5}
\|(\Phi,\Psi)(\tau)\|_{H^{2}}^{2}+\int_{0}^{\tau}\int \mid\tilde{U_{y}}\mid \Psi^{2}dyds+\int_{0}^{\tau}\|\Phi_{y}(s)\|_{H^{1}}^{2}+\|\Psi_{y}(s)\|_{H^{2}}^{2}ds\leq C\delta^{\frac{1}{3}}.
\end{equation}
\end{proposition}

\section{Proof of Proposition 3.2}

In this section, we prove Proposition 3.2.
It can be obtained by the following lemmas. From the a priori assumption \eqref{387.125} with  $\eta_{0}>0$  small enough, it holds that
 \begin{equation}
\frac{1}{4}v_{+}\leq v\leq2v_{+}.
\end{equation}

\begin{lemma}
If $\delta_{0}$ and $\eta_{0}$ are suitable small, for $0\leq \tau\leq T$, it holds for $\delta\leq\delta_{0}$ that
\begin{equation}\label{39.5}
\begin{array}{ll}
\displaystyle \|(\Phi,\Psi)(\tau)\|^{2}+\int_{0}^{\tau}\int \mid\tilde{U_{y}}\mid \Psi^{2}dyds+\int_{0}^{\tau}\|\Psi_{y}(s)\|^{2}ds\\[5mm]
\leq C\delta
\displaystyle+C(\eta_{0}^{\frac{1}{2}}+\delta^{2})\int_{0}^{\tau}\|\Phi_{y}(s)\|^{2}ds+C\eta_{0}^{\frac{1}{2}}\int_{0}^{\tau}\|\Psi_{yy}(s)
\|^{2}ds.
\end{array}
\end{equation}
\end{lemma}
\begin{proof}
Multiplying $\eqref{351221.6}_{1}$ by $\Phi$, $\eqref{351221.6}_{2}$ by $-\dfrac{\Psi}{p'(\tilde{V})}$, we obtain
\begin{equation}\label{332769.5}
\big(\frac{1}{2}\Phi^{2}\big)_{\tau}-\Phi\Psi_{y}=0,
\end{equation}
and
\begin{equation}\label{33sd2769.5}
-\Psi_{\tau}\Psi\dfrac{1}{p'(\tilde{V})}-\Psi\Phi_{y}=\dfrac{\Psi\Psi_{yy}}{\tilde{V}^{1+\alpha}|p'(\tilde{V})|}-
\dfrac{\Psi}{p'(\tilde{V})}Q.
\end{equation}
Since
\begin{equation}\label{33s.5}
\begin{array}{ll}
-\Psi_{\tau}\Psi\dfrac{1}{p'(\tilde{V})}-\Psi\Phi_{y}
&=-\big(\dfrac{1}{2p'(\tilde{V})}\Psi^{2}\big)_{\tau}+\big(\dfrac{1}{2p'(\tilde{V})}\big)_{\tau}\Psi^{2}-\Psi\Phi_{y}\\[5mm]
&=\big(\dfrac{1}{2|p'(\tilde{V})|}\Psi^{2}\big)_{\tau}-\dfrac{p''(\tilde{V})}{2(p'(\tilde{V}))^{2}}\tilde{V}\tau\Psi^{2}-\Psi\Phi_{y},
\end{array}
\end{equation}
and
\begin{equation}\label{33sd25}
\begin{array}{ll}
\dfrac{\Psi\Psi_{yy}}{\tilde{V}^{1+\alpha}|p'(\tilde{V})|}-\dfrac{\Psi}{p'(\tilde{V})}Q\\[5mm]
=\big(\dfrac{\Psi\Psi_{y}}{\tilde{V}^{1+\alpha}|p'(\tilde{V})|}\big)_{y}-\dfrac{\Psi_{y}^{2}}{\tilde{V}^{1+\alpha}|p'(\tilde{V})|}-
\big(\dfrac{1}{\tilde{V}^{1+\alpha}|p'(\tilde{V})|}\big)_{y}\Psi\Psi_{y}+\dfrac{\Psi}{|p'(\tilde{V})|}Q.
\end{array}
\end{equation}

Substituting \eqref{33s.5} and \eqref{33sd25} into \eqref{33sd2769.5} shows
\begin{equation}\label{366769.5}
\begin{array}{ll}
\big(\dfrac{1}{2|p'(\tilde{V})|}\Psi^{2}\big)_{\tau}-\dfrac{p''(\tilde{V})}{2|p'(\tilde{V})|^{2}}\tilde{V}\tau\Psi^{2}-\Phi_{y}\Psi\\[5mm]
=\big(\dfrac{\Psi\Psi_{y}}{\tilde{V}^{1+\alpha}|p'(\tilde{V})|}\big)_{y}-\dfrac{\Psi_{y}^{2}}{\tilde{V}^{1+\alpha}|p'(\tilde{V})|}-
\big(\dfrac{1}{\tilde{V}^{1+\alpha}|p'(\tilde{V})|}\big)_{y}\Psi\Psi_{y}+\dfrac{\Psi}{|p'(\tilde{V})|}Q.
\end{array}
\end{equation}

Adding the above two equations \eqref{332769.5}, \eqref{366769.5} then integrating it,  we have
\begin{equation}\label{36saa6769.5}
\begin{array}{ll}
\displaystyle\int^{\tau}_{0}\int(\frac{1}{2}\Phi^{2}+\dfrac{1}{2|p'(\tilde{V})|}\Psi^{2})_{\tau}dyds-\int^{\tau}_{0}\int\dfrac{p''(\tilde{V})}{2|p'(\tilde{V})|^{2}}
\displaystyle\tilde{V}\tau\Psi^{2}dyds+\int^{\tau}_{0}\int\dfrac{\Psi_{y}^{2}}{\tilde{V}^{1+\alpha}|p'(\tilde{V})|}dyds\\[5mm]
\displaystyle =\int^{\tau}_{0}\int\big(\dfrac{\Psi\Psi_{y}}{\tilde{V}^{1+\alpha}|p'(\tilde{V})|}+
\displaystyle \Phi\Psi\big)_{y}dyds-\int^{\tau}_{0}\int\big(\dfrac{1}{\tilde{V}^{1+\alpha}|p'(\tilde{V})|}\big)_{y}\Psi\Psi_{y}dyds\\[5mm]
\displaystyle +\int^{\tau}_{0}\int\dfrac{\Psi}{|p'(\tilde{V})|}Qdyds.
\end{array}
\end{equation}

We now estimate the terms in the above formula. Applying \eqref{33121.6} and Proposition 2.1, one can get
\begin{equation}\label{3769.5}
\begin{array}{ll}
\displaystyle-\int^{\tau}_{0}\int\dfrac{p''(\tilde{V})}{2|p'(\tilde{V})|^{2}}\tilde{V}\tau\Psi^{2}dyds
=\displaystyle\int^{\tau}_{0}\int\dfrac{p''(\tilde{V})}{2|p'(\tilde{V})|^{2}}|\tilde{U}_{y}|\Psi^{2}dyds,
\end{array}
\end{equation}
\begin{equation}\label{37ds69.5}
\begin{array}{ll}
\displaystyle\int^{\tau}_{0}\int\big(\dfrac{\Psi\Psi_{y}}{\tilde{V}^{1+\alpha}|p'(\tilde{V})|}+\Phi\Psi\big)_{y}dyds=0,
\end{array}
\end{equation}
and
\begin{equation}
\begin{array}{ll}\label{3454}
\displaystyle -\int^{\tau}_{0}\int\big(\dfrac{1}{\tilde{V}^{1+\alpha}|p'(\tilde{V})|}\big)_{y}\Psi\Psi_{y}dyds
&\displaystyle\leq C\int^{\tau}_{0}\int|\tilde{V}_{y}\Psi\Psi_{y}|dyds\\[4mm]
&\displaystyle\leq C\delta\int^{\tau}_{0}\int|\tilde{U}_{y}|^{\frac{1}{2}}|\Psi\Psi_{y}|dyds\\[4mm]
&\displaystyle\leq C\delta\int^{\tau}_{0}\int\Psi_{y}^{2}dyds+C\delta\int^{\tau}_{0}\int|\tilde{U}_{y}|\Psi^{2}dyds,
\end{array}
\end{equation}
here we have used the fact that
\begin{equation}
|\tilde{V}_{y}|=|\frac{1}{s}\tilde{V}_{\tau}|=|\frac{1}{s}\tilde{U}_{y}|
\leq C\delta|\tilde{U}_{y}|^{\frac{1}{2}}.
\end{equation}

It follows from \eqref{37.15} that
\begin{equation}\label{33219.5}
\begin{array}{ll}
\displaystyle -\int^{\tau}_{0}\int\dfrac{\Psi}{p'(\tilde{V})}Qdyds
\displaystyle\leq C\int^{\tau}_{0}\int|\Psi Q|dyds\\[5mm]
\displaystyle\leq C\int^{\tau}_{0}\int|\Psi\Phi^{2}_{y}|dyds+C\int^{\tau}_{0}\int
\displaystyle|\Psi\Phi_{y}\Psi_{yy}|dyds+C\delta^{2}\int^{\tau}_{0}\int e^{-c\delta|y|-c\delta|s|}|\Psi\Phi_{y}|dyds\\[5mm]
\displaystyle=:\sum\limits_{i=1}^{3} I_i.
\end{array}
\end{equation}

Using \eqref{387.125}, Cauchy inequality and Sobolev's interpolation inequality, we have
\begin{equation}
\begin{array}{ll}
\hspace{-3.6cm}\displaystyle I_{1}\leq\eta_{0}^{\frac{1}{2}}\int^{\tau}_{0}\|\Phi_{y}\|^{2}ds,
\end{array}
\end{equation}

\begin{equation}
\begin{array}{ll}
\displaystyle I_{2}\leq C\eta_{0}^{\frac{1}{2}}\int^{\tau}_{0}\|\Phi_{y}\|^{2}ds+C\eta_{0}^{\frac{1}{2}}\int^{\tau}_{0}\|\Psi_{yy}\|^{2}ds,
\end{array}
\end{equation}
and
\begin{equation}
\begin{array}{ll}
\displaystyle I_{3}&\displaystyle\leq C\delta^{2}\int^{\tau}_{0}\int e^{-2c\delta|y|-2c\delta|s|}\Psi^{2}dyds+
\displaystyle C\delta^{2}\int^{\tau}_{0}\int \Phi_{y}^{2}dyds\\[5mm]
&\displaystyle\leq C\delta^{2}\int^{\tau}_{0}e^{-2c\delta|s|}\|\Psi\|^{2}_{L^{\infty}}ds\int e^{-2c\delta|y|}dy
+C\delta^{2}\int^{\tau}_{0}\int \Phi_{y}^{2}dyds\\[5mm]
&\displaystyle\leq C\delta\int^{\tau}_{0}e^{-2c\delta|s|}\|\Psi\|\|\Psi_{y}\|ds
+C\delta^{2}\int^{\tau}_{0}\|\Phi_{y}\|^{2}ds\\[5mm]
&\displaystyle\leq C\delta\eta_{0}^{\frac{1}{2}}\int^{\tau}_{0}e^{-2c\delta|s|}\|\Psi_{y}\|ds
+C\delta^{2}\int^{\tau}_{0}\|\Phi_{y}\|^{2}ds\\[5mm]
&\displaystyle\leq C\delta^{2}\int^{\tau}_{0}e^{-4c\delta|s|}ds+C\eta_{0}\int^{\tau}_{0}\|\Psi_{y}\|^{2}ds
+C\delta^{2}\int^{\tau}_{0}\|\Phi_{y}\|^{2}ds\\[5mm]
&\displaystyle\leq C\delta+C\eta_{0}\int^{\tau}_{0}\|\Psi_{y}\|^{2}ds
+C\delta^{2}\int^{\tau}_{0}\|\Phi_{y}\|^{2}ds.
\end{array}
\end{equation}
Thus,
\begin{equation}\label{2121}
\begin{array}{ll}
\displaystyle -\int^{\tau}_{0}\int\dfrac{\Psi}{p'(\tilde{V})}Qdyds\\[5mm]
\leq
\displaystyle C(\eta_{0}^{\frac{1}{2}}+\delta^{2})\int^{\tau}_{0}\|\Phi_{y}\|^{2}ds+C\eta_{0}\int^{\tau}_{0}\|\Psi_{y}\|^{2}ds
\displaystyle +C\eta_{0}^{\frac{1}{2}}\int^{\tau}_{0}\|\Psi_{yy}\|^{2}ds+C\delta.
\end{array}
\end{equation}

Now returning to \eqref{36saa6769.5}, we can easily obtain the estimate \eqref{39.5}.
\end{proof}
\begin{lemma}
If $\delta_{0}$ and $\eta_{0}$ are suitable small, for  $0\leq \tau\leq T$, it holds for $\delta\leq\delta_{0}$ that
\begin{equation}\label{390.5}
\begin{array}{ll}
\displaystyle \|\Phi(\tau)\|_{H^{1}}^{2}+\|\Psi(\tau)\|^{2}+\int_{0}^{\tau}\int |\tilde{U_{y}}| \Psi^{2}dyds+\int_{0}^{\tau}\|\Phi_{y}(s)\|^{2}ds+\int_{0}^{\tau}\|\Psi_{y}(s)\|^{2}ds\\[5mm]
\leq C\delta
\displaystyle+C\eta_{0}^{\frac{1}{2}}\int_{0}^{\tau}\|\Psi_{yy}(s)
\|^{2}ds.
\end{array}
\end{equation}
\end{lemma}
\begin{proof}
It follows from $\eqref{351221.6}_{1}$ that $\Psi_{y}=\Phi_{\tau}$. Substituting it into $\eqref{351221.6}_{2}$, we can obtain
\begin{equation}\label{34590.5}
\begin{array}{ll}
\dfrac{1}{\tilde{V}^{1+\alpha}}\Phi_{y\tau}-p'(\tilde{V})\Phi_{y}-\Psi_{\tau}=-Q.
\end{array}
\end{equation}
Multiplying \eqref{34590.5} by $\Phi_{y}$, then integrating the result formula, we have
\begin{equation}\label{345210.5}
\begin{array}{ll}
\displaystyle\int_{0}^{\tau}\int\frac{1}{2\tilde{V}^{1+\alpha}}(\Phi_{y}^{2})_{\tau}dyds-\int_{0}^{\tau}\int p'(\tilde{V})\Phi_{y}^{2}dyds
\displaystyle-\int_{0}^{\tau}\int\Psi_{\tau}\Phi_{y}dyds
\displaystyle=-\int_{0}^{\tau}\int Q\Phi_{y}dyds.
\end{array}
\end{equation}

The left side of \eqref{345210.5} is
\begin{equation}\label{36210.5}
\begin{array}{ll}
\displaystyle\int_{0}^{\tau}\int(\frac{1}{2\tilde{V}^{1+\alpha}}\Phi_{y}^{2}-\Psi\Phi_{y})_{\tau}dyds-
\displaystyle\int_{0}^{\tau}\int(\frac{1}{2\tilde{V}^{1+\alpha}})_{\tau}\Phi_{y}^{2}dyds
\displaystyle+\int_{0}^{\tau}\int\Psi\Phi_{y\tau}dyds\\[5mm]
\displaystyle+\int_{0}^{\tau}\int|p'(\tilde{V})|\Phi_{y}^{2}dyds\\[5mm]
\displaystyle=\int(\frac{1}{2\tilde{V}^{1+\alpha}}\Phi_{y}^{2}-\Psi\Phi_{y})dy-
\displaystyle\int_{0}^{\tau}\int(\frac{1}{2\tilde{V}^{1+\alpha}})_{\tau}\Phi_{y}^{2}dyds
\displaystyle-\int_{0}^{\tau}\int\Psi_{y}^{2}dyds\\[5mm]
\displaystyle+\int_{0}^{\tau}\int|p'(\tilde{V})|\Phi_{y}^{2}dyds,
\end{array}
\end{equation}
then we deduce
\begin{equation}\label{30-.5}
\begin{array}{ll}
\displaystyle\int\frac{1}{2\tilde{V}^{1+\alpha}}\Phi_{y}^{2}dy+\int_{0}^{\tau}\int|p'(\tilde{V})|\Phi_{y}^{2}dyds\\[5mm]
=\displaystyle\int_{0}^{\tau}\int\Psi_{y}^{2}dyds+\int\Psi\Phi_{y}dy+\int_{0}^{\tau}\int(\frac{1}{2\tilde{V}^{1+\alpha}})_{\tau}\Phi_{y}^{2}dyds
\displaystyle-\int_{0}^{\tau}\int Q\Phi_{y}dyds\\[5mm]
=:\displaystyle\int_{0}^{\tau}\int\Psi_{y}^{2}dyds+\sum\limits_{i=4}^{6} I_i.
\end{array}
\end{equation}

Employing  \eqref{37.15}, Proposition 2.1, $\epsilon$-Young inequality and the Sobolev interpolation inequality, we can get the estimates of $I_{4}$-$I_{6}$, respectively.
\begin{equation}\label{310-.5}
\begin{array}{ll}
I_{4}\leq\epsilon\|\Phi_{y}\|^{2}+C\|\Psi\|^{2},
\end{array}
\end{equation}
where $\epsilon$ is a suitable small and fixed constant.
\begin{equation}\label{3010-.5}
\begin{array}{ll}
\displaystyle I_{5}\leq C\delta^{2}\int_{0}^{\tau}\|\Phi_{y}\|^{2}ds,
\end{array}
\end{equation}
and
\begin{equation}\label{3010-.5}
\begin{array}{ll}
\displaystyle I_{6}&\displaystyle\leq C\int_{0}^{\tau}\int|\Phi_{y}|^{3}dyds+C\int_{0}^{\tau}\int |\Psi_{yy}|\Phi_{y}^{2}dyds+C\delta^{2}\int_{0}^{\tau}\int e^{-c\delta|y|-c\delta|s|}\Phi_{y}^{2}dyds\\[5mm]
&\displaystyle\leq C(\eta_{0}^{\frac{1}{2}}+\delta^{2})\int_{0}^{\tau}\|\Phi_{y}\|^{2}ds+C\eta_{0}^{\frac{1}{2}}\int_{0}^{\tau}\|\Psi_{yy}\|^{2}ds.
\end{array}
\end{equation}

Taking the above estimates into \eqref{30-.5}, we have
\begin{equation}\label{3010--.5}
\begin{array}{ll}
\displaystyle\|\Phi_{y}\|^{2}+\int_{0}^{\tau}\|\Phi_{y}\|^{2}ds\\[5mm]
\displaystyle\leq \epsilon\|\Phi_{y}\|^{2}+C (\epsilon)\|\Psi\|^{2}+
\displaystyle C(\eta_{0}^{\frac{1}{2}}+\delta^{2})\int_{0}^{\tau}\|\Phi_{y}\|^{2}ds+\int_{0}^{\tau}\|\Psi_{y}\|^{2}ds
\displaystyle +C\eta_{0}^{\frac{1}{2}}\int_{0}^{\tau}\|\Psi_{yy}\|^{2}ds.
\end{array}
\end{equation}
By using the smallness of $\epsilon,\delta$ and $\eta_{0}$, we have
\begin{equation}\label{3010--.5}
\begin{array}{ll}
\displaystyle\|\Phi_{y}\|^{2}+\int_{0}^{\tau}\|\Phi_{y}\|^{2}ds
\leq \displaystyle C\|\Psi\|^{2}
+\int_{0}^{\tau}\|\Psi_{y}\|^{2}ds+C\eta_{0}^{\frac{1}{2}}\int_{0}^{\tau}\|\Psi_{yy}\|^{2}ds.
\end{array}
\end{equation}
Then combining Lemma 4.1 with \eqref{3010--.5}, we complete the proof of Lemma 4.2.
\end{proof}

\begin{lemma}
If $\delta_{0}$ and $\eta_{0}$ are suitable small, for  $0\leq \tau\leq T$, it holds for $\delta\leq\delta_{0}$ that
\begin{equation}\label{3090.5}
\begin{array}{ll}
\displaystyle \|\phi\|_{H^{1}}^{2}+\|\psi\|^{2}+\int_{0}^{\tau}\|\phi_{y}(s)\|^{2}ds+\int_{0}^{\tau}\|\psi_{y}(s)\|^{2}ds\leq C\delta^{\frac{1}{3}}+C\eta_{0}^{\frac{1}{2}}\int_{0}^{\tau}\|\psi_{yy}\|^{2}ds.
\end{array}
\end{equation}
\end{lemma}
\begin{proof}
Differentiate \eqref{3521.6} with respect to $y$ gives
\begin{equation}\label{3=-09}
\left\{\begin{array}{l}
\phi_{\tau}-\psi_{y}=0,\\
\psi_{\tau}+p'(v)v_{y}-p'(\tilde{V})\tilde{V}_{y}=(\dfrac{u_{y}}{v^{1+\alpha}})_{y}-(\dfrac{\tilde{U}_{y}}{\tilde{V}^{1+\alpha}})_{y},
\end{array}\right.
\end{equation}
where
\begin{equation}\label{3==-09}
\begin{array}{ll}
p'(v)v_{y}-p'(\tilde{V})\tilde{V}_{y}=(p'(v)-p'(\tilde{V}))v_{y}+p'(\tilde{V})\phi_{y},
\end{array}
\end{equation}
and
\begin{equation}\label{3==-09}
\begin{array}{ll}
(\dfrac{u_{y}}{v^{1+\alpha}})_{y}-(\dfrac{\tilde{U}_{y}}{\tilde{V}^{1+\alpha}})_{y}
&=(\dfrac{\psi_{y}}{v^{1+\alpha}})_{y}+[\tilde{U}_{y}(\dfrac{1}{v^{1+\alpha}}-\dfrac{1}{\tilde{V}^{1+\alpha}})]_{y}.
\end{array}
\end{equation}
Thus \eqref{3=-09} can be rewritten as
\begin{equation}\label{3=-0-09}
\left\{\begin{array}{l}
\phi_{\tau}-\psi_{y}=0,\\
\psi_{\tau}+p'(\tilde{V})\phi_{y}=(\dfrac{\psi_{y}}{v^{1+\alpha}})_{y}+Q_{1},
\end{array}\right.
\end{equation}
where
\begin{equation}\label{3==}
\begin{array}{l}
Q_{1}=(\tilde{U}_{y}(\dfrac{1}{v^{1+\alpha}}-\dfrac{1}{\tilde{V}^{1+\alpha}}))_{y}-(p'(v)-p'(\tilde{V}))v_{y}.
\end{array}
\end{equation}

Multiplying $\eqref{3=-0-09}_{1}$ by $\phi$, $\eqref{3=-0-09}_{2}$ by $\frac{\psi}{|p'(\tilde{V})|}$ and adding them up, then integrating the result formula on
$[0,\tau]\times R$, we have
\begin{equation}\label{3==}
\begin{array}{l}
\displaystyle \int_{0}^{\tau}\int (\phi\phi_{\tau}-\phi\psi_{y}+\frac{1}{|p'(\tilde{V})|}\psi\psi_{\tau}-\psi\phi_{y})dyds\\[5mm]
\displaystyle=\int_{0}^{\tau}\int\frac{\psi}{|p'(\tilde{V})|}(\dfrac{\psi_{y}}{v^{1+\alpha}})_{y}dyds+
\int_{0}^{\tau}\int\frac{\psi}{|p'(\tilde{V})|}Q_{1}dyds.
\end{array}
\end{equation}

The left side of \eqref{3==} is
\begin{equation}\label{3=12=}
\begin{array}{l}
\displaystyle\int_{0}^{\tau}\int (\phi\phi_{\tau}-\phi\psi_{y}+\frac{1}{|p'(\tilde{V})|}\psi\psi_{\tau}-\psi\phi_{y})dyds\\[5mm]
\displaystyle=\frac{1}{2}\int (\phi^{2}+\frac{1}{|p'(\tilde{V})|}\psi^{2})dy-
\displaystyle\int_{0}^{\tau}\int\frac{p''(\tilde{V})}{2|p'(\tilde{V})|^{2}}\tilde{V}_{\tau}\psi^{2}dyds.
\end{array}
\end{equation}

Using the integration by parts, one can get
\begin{equation}\label{34=12=}
\begin{array}{l}
\displaystyle\int_{0}^{\tau}\int\frac{\psi}{|p'(\tilde{V})|}(\dfrac{\psi_{y}}{v^{1+\alpha}})_{y}dyds+
\displaystyle\int_{0}^{\tau}\int\frac{\psi}{|p'(\tilde{V})|}Q_{1}dyds\\[5mm]
\displaystyle=-\int_{0}^{\tau}\int\frac{1}{|p'(\tilde{V})|}\dfrac{1}{v^{1+\alpha}}\psi_{y}^{2}dyds-
\displaystyle\int_{0}^{\tau}\int(\frac{1}{|p'(\tilde{V})|})_{y}\dfrac{1}{v^{1+\alpha}}\psi\psi_{y}dyds
\displaystyle+\int_{0}^{\tau}\int\frac{\psi}{|p'(\tilde{V})|}Q_{1}dyds.

\end{array}
\end{equation}

Putting the above two formula to \eqref{3==} leads to
\begin{equation}\label{3ew}
\begin{array}{l}
\displaystyle\frac{1}{2}\int (\phi^{2}+\frac{1}{|p'(\tilde{V})|}\psi^{2})dy+
\displaystyle\int_{0}^{\tau}\int\frac{1}{|p'(\tilde{V})|}\dfrac{1}{v^{1+\alpha}}\psi_{y}^{2}dyds\\[5mm]
\displaystyle=\int_{0}^{\tau}\int\frac{p''(\tilde{V})}{2|p'(\tilde{V})|^{2}}\tilde{V}_{\tau}\psi^{2}dyds-
\displaystyle\int_{0}^{\tau}\int(\frac{1}{|p'(\tilde{V})|})_{y}\dfrac{1}{v^{1+\alpha}}\psi\psi_{y}dyds
\displaystyle+\int_{0}^{\tau}\int\frac{\psi}{|p'(\tilde{V})|}Q_{1}dyds\\[5mm]
\displaystyle=:\sum\limits_{i=7}^{9} I_i.
\end{array}
\end{equation}

Using Propostion 2.1 and Cauchy inequality, we have
\begin{equation}\label{3eww}
\begin{array}{l}
\displaystyle I_{7}\leq C\delta^{2}\int_{0}^{\tau}\int\psi^{2}dyds,
\end{array}
\end{equation}
and
\begin{equation}\label{332eww}
\begin{array}{l}
\displaystyle I_{8}\leq C\delta^{2}\int_{0}^{\tau}\int\psi^{2}dyds+C\delta^{2}\int_{0}^{\tau}\int\psi_{y}^{2}dyds.
\end{array}
\end{equation}

It follows from the  integration by parts that
\begin{equation}\label{3e-02ww}
\begin{array}{l}
\displaystyle I_{9}
\displaystyle=-\int_{0}^{\tau}\int(\frac{1}{|p'(\tilde{V})|})_{y}\psi\tilde{U}_{y}(\dfrac{1}{v^{1+\alpha}}-\dfrac{1}{\tilde{V}^{1+\alpha}})dyds
\displaystyle-\int_{0}^{\tau}\int\frac{1}{|p'(\tilde{V})|}\psi_{y}\tilde{U}_{y}(\dfrac{1}{v^{1+\alpha}}-\dfrac{1}{\tilde{V}^{1+\alpha}})dyds\\[5mm]
\displaystyle-\int_{0}^{\tau}\int\frac{\psi}{|p'(\tilde{V})|}(p'(v)-p'(\tilde{V}))(\phi_{y}+\tilde{V}_{y})dyds\\[5mm]
\displaystyle=:\sum\limits_{i=1}^{3} K_i.
\end{array}
\end{equation}
\begin{equation}\label{3df}
\begin{array}{ll}
\displaystyle K_{1}&\displaystyle\leq C\int_{0}^{\tau}\int|\tilde{V}_{y}\psi\tilde{U}_{y}\phi|dyds
\displaystyle\leq C\delta^{4}\int_{0}^{\tau}\int|\psi\phi|dyds\\[5mm]
\displaystyle&\displaystyle\leq C\delta^{4}\int_{0}^{\tau}\int\psi^{2}dyds+C\delta^{4}\int_{0}^{\tau}\int\phi^{2}dyds,
\end{array}
\end{equation}
\begin{equation}\label{3drf}
\begin{array}{ll}
\displaystyle K_{2}&\displaystyle\leq C\delta^{2}\int_{0}^{\tau}\int|\psi_{y}\phi|dyds
\displaystyle\leq C\delta^{2}\int_{0}^{\tau}\int\psi_{y}^{2}dyds+C\delta^{2}\int_{0}^{\tau}\int\phi^{2}dyds,
\end{array}
\end{equation}
and
\begin{equation}\label{3dfrf}
\begin{array}{ll}
\displaystyle K_{3}&\displaystyle\leq C\int_{0}^{\tau}\int|\psi(\phi_{y}+\tilde{V}_{y})|dyds\\[5mm]
&\displaystyle\leq C\int_{0}^{\tau}\int|\psi\phi_{y}|dyds+ C\int_{0}^{\tau}\int|\psi\tilde{V}_{y}|dyds\\[5mm]
&\displaystyle\leq \epsilon\int_{0}^{\tau}\|\phi_{y}\|^{2}ds+C\int_{0}^{\tau}\|\psi\|^{2}ds+C\delta^{2}\int_{0}^{\tau}
\displaystyle e^{-c\delta|s|}\|\psi\|_{L^{\infty}}ds\int e^{-c\delta|y|}dy\\[5mm]
&\displaystyle\leq \epsilon\int_{0}^{\tau}\|\phi_{y}\|^{2}ds+C\int_{0}^{\tau}\|\psi\|^{2}ds+
\displaystyle C\delta\int_{0}^{\tau}e^{-c\delta|s|}\|\psi\|^{\frac{1}{2}}|\psi_{y}\|^{\frac{1}{2}}ds\\[5mm]
&\displaystyle\leq \epsilon\int_{0}^{\tau}\|\phi_{y}\|^{2}ds+C\int_{0}^{\tau}\|\psi\|^{2}ds+
\displaystyle\sup_{\tau\in[0,T]}\|\psi\|^{2}\int_{0}^{\tau}\|\psi_{y}\|^{2}ds+C\delta^{\frac{4}{3}}
\displaystyle\int_{0}^{\tau}e^{-\frac{4}{3}c\delta|s|}ds\\[5mm]
&\displaystyle\leq\epsilon\int_{0}^{\tau}\|\phi_{y}\|^{2}ds+C\int_{0}^{\tau}\|\psi\|^{2}ds+\eta_{0}
\int_{0}^{\tau}\|\psi_{y}\|^{2}ds+C\delta^{\frac{1}{3}},
\end{array}
\end{equation}
where $\epsilon$ is a  suitable small and fixed constant.

Putting $K_1,K_2$ and $K_3$ into \eqref{3e-02ww}, we can get that
\begin{equation}\label{3df=rf}
\begin{array}{ll}
\displaystyle I_{9}&\displaystyle\leq C(1+\delta^{4})\int_{0}^{\tau}\|\psi\|^{2}ds+C\delta^{2}\int_{0}^{\tau}\|\phi\|^{2}ds
\displaystyle+C(\eta_{0}+\delta^{2})\int_{0}^{\tau}\|\psi_{y}\|^{2}ds\\[5mm]
&\displaystyle+\epsilon\int_{0}^{\tau}\|\phi_{y}\|^{2}ds+C\delta^{\frac{1}{3}}.
\end{array}
\end{equation}

Returning to \eqref{3ew} and applying Lemma 4.2, one can obtain
\begin{equation}\label{3r2=2f}
\begin{array}{ll}
\displaystyle\|\phi\|^{2}+\|\psi\|^{2}+\int_{0}^{\tau}\|\psi_{y}\|^{2}ds
\displaystyle\leq C\delta^{\frac{1}{3}}+\epsilon\int_{0}^{\tau}\|\phi_{y}\|^{2}ds.
\end{array}
\end{equation}

On the other hand, substituting $\eqref{3=-0-09}_{1}$ into $\eqref{3=-0-09}_{2}$ gives
\begin{equation}\label{3df=r}
\begin{array}{ll}
\psi_{\tau}+p'(\tilde{V})\phi_{y}
&=(\psi_{y}(\dfrac{1}{v^{1+\alpha}}-\dfrac{1}{\tilde{V}^{1+\alpha}}))_{y}+(\dfrac{\psi_{y}}{\tilde{V}^{1+\alpha}})_{y}+Q_{1}\\[5mm]
&=(\psi_{y}(\dfrac{1}{v^{1+\alpha}}-\dfrac{1}{\tilde{V}^{1+\alpha}}))_{y}+\dfrac{\phi_{y\tau}}{\tilde{V}^{1+\alpha}}+
(\dfrac{1}{\tilde{V}^{1+\alpha}})_{y}\psi_{y}+Q_{1},
\end{array}
\end{equation}
i.e.
\begin{equation}\label{316df=r}
\begin{array}{ll}
\dfrac{\phi_{y\tau}}{\tilde{V}^{1+\alpha}}-\psi_{\tau}-p'(\tilde{V})\phi_{y}=-(\dfrac{1}{\tilde{V}^{1+\alpha}})_{y}\psi_{y}-
(\psi_{y}(\dfrac{1}{v^{1+\alpha}}-\dfrac{1}{\tilde{V}^{1+\alpha}}))_{y}-Q_{1}.
\end{array}
\end{equation}

Multiplying  \eqref{316df=r} by $\phi_{y}$,  we have
\begin{equation}\label{316fr}
\begin{array}{ll}
\displaystyle\int_{0}^{\tau}\int\dfrac{1}{2\tilde{V}^{1+\alpha}}(\phi_{y}^{2})_{\tau}-\phi_{y}\psi_{\tau}-p'(\tilde{V})\phi_{y}^{2}dyds\\[5mm]
\displaystyle=\int_{0}^{\tau}\int-(\dfrac{1}{\tilde{V}^{1+\alpha}})_{y}\psi_{y}\phi_{y}-
(\psi_{y}(\dfrac{1}{v^{1+\alpha}}-\dfrac{1}{\tilde{V}^{1+\alpha}}))_{y}\phi_{y}-Q_{1}\phi_{y}dyds.
\end{array}
\end{equation}

Integrating by parts and using $\eqref{3=-0-09}_{1}$ yield
\begin{equation}\label{36fr}
\begin{array}{ll}
\displaystyle\int_{0}^{\tau}\int\dfrac{1}{2\tilde{V}^{1+\alpha}}(\phi_{y}^{2})_{\tau}-\phi_{y}\psi_{\tau}-p'(\tilde{V})\phi_{y}^{2}dyds\\[5mm]
\displaystyle=\int\dfrac{\phi_{y}^{2}}{2\tilde{V}^{1+\alpha}}dy-
\displaystyle\int_{0}^{\tau}\int(\dfrac{1}{2\tilde{V}^{1+\alpha}})_{\tau}\phi_{y}^{2}dyds-\int\phi_{y}\psi dy\\[5mm]
\displaystyle+\int_{0}^{\tau}\int\psi_{yy}\psi dyds+\int_{0}^{\tau}\int|p'(\tilde{V})|\phi_{y}^{2}dyds.
\end{array}
\end{equation}
Returning to \eqref{316fr}, we have
\begin{equation}\label{36fr}
\begin{array}{ll}
\displaystyle\int(\dfrac{\phi_{y}^{2}}{2\tilde{V}^{1+\alpha}}-\phi_{y}\psi)dy+\int_{0}^{\tau}\int|p'(\tilde{V})|\phi_{y}^{2}dyds\\[5mm]
\displaystyle=-\int_{0}^{\tau}\int\psi_{yy}\psi dyds+\int_{0}^{\tau}\int(\dfrac{1}{2\tilde{V}^{1+\alpha}})_{\tau}\phi_{y}^{2}dyds-
\int_{0}^{\tau}\int(\dfrac{1}{\tilde{V}^{1+\alpha}})_{y}\psi_{y}\phi_{y}dyds\\[5mm]
\displaystyle-\int_{0}^{\tau}\int(\psi_{y}(\dfrac{1}{v^{1+\alpha}}-\dfrac{1}{\tilde{V}^{1+\alpha}}))_{y}\phi_{y}dyds-
\int_{0}^{\tau}\int Q_{1}\phi_{y}dyds\\[5mm]
=:\sum\limits_{i=10}^{14}I_{i}.
\end{array}
\end{equation}
\begin{equation}\label{3r}
\begin{array}{ll}
\displaystyle I_{10}=\int_{0}^{\tau}\int\psi_{y}^{2}dyds,
\end{array}
\end{equation}
\begin{equation}\label{3ttr}
\begin{array}{ll}
\displaystyle I_{11}\leq C\delta^{2}\int_{0}^{\tau}\int\phi_{y}^{2}dyds,
\end{array}
\end{equation}
\begin{equation}\label{3ttr}
\begin{array}{ll}
\displaystyle I_{12}
\leq C\delta^{2}\int_{0}^{\tau}\int\psi_{y}^{2}dyds+C\delta^{2}\int_{0}^{\tau}\int\phi_{y}^{2}dyds.

\end{array}
\end{equation}
and
\begin{equation}\label{3ttr}
\begin{array}{ll}
\displaystyle I_{13}\displaystyle=-\int_{0}^{\tau}\int\psi_{yy}(\dfrac{1}{v^{1+\alpha}}-\dfrac{1}{\tilde{V}^{1+\alpha}})\phi_{y}dyds
\displaystyle-\int_{0}^{\tau}\int\psi_{y}(\dfrac{1}{v^{1+\alpha}}-\dfrac{1}{\tilde{V}^{1+\alpha}})_{y}\phi_{y}dyds\\[5mm]
\displaystyle=-(1+\alpha)\int_{0}^{\tau}\int\frac{\bar{v}^{\alpha}}{(v\tilde{V})^{1+\alpha}}\phi\phi_{y}\psi_{yy}dyds-
(1+\alpha)\int_{0}^{\tau}\int(\frac{\bar{v}^{\alpha}\phi}{(v\tilde{V})^{1+\alpha}})_{y}\phi_{y}\psi_{y}dyd\\[5mm]
=:K_{4}+K_{5}
\end{array}
\end{equation}
Since
\begin{equation}\label{3t8tr}
\begin{array}{ll}
\displaystyle K_{4}\leq C\eta_{0}^{\frac{1}{2}}\int_{0}^{\tau}\|\phi_{y}\|^{2}ds+C\eta_{0}^{\frac{1}{2}}\int_{0}^{\tau}\|\psi_{yy}\|^{2}ds,
\end{array}
\end{equation}
and
\begin{equation}\label{32t8tr}
\begin{array}{ll}
\displaystyle K_{5}&\displaystyle\leq C\int_{0}^{\tau}\int|\psi_{y}\phi\phi_{y}^{2}|+|\tilde{V}_{y}\psi_{y}\phi\phi_{y}|+
|\psi_{y}\phi_{y}^{2}|dyds\\[5mm]
&\displaystyle \leq C\eta_{0}^{\frac{1}{2}}\int_{0}^{\tau}\|\psi_{y}\|_{L^{\infty}}\|\phi_{y}\|^{2}ds+
C\delta^{2}\eta_{0}^{\frac{1}{2}}\int_{0}^{\tau}\int|\psi_{y}\phi_{y}|dyds+C\int_{0}^{\tau}\|\psi_{y}\|_{L^{\infty}}\|\phi_{y}\|^{2}ds\\[5mm]
&\displaystyle\leq C\eta_{0}^{\frac{1}{2}}\int_{0}^{\tau}\|\psi_{y}\|^{\frac{1}{2}}\|\psi_{yy}\|^{\frac{1}{2}}
\|\phi_{y}\|^{\frac{1}{2}}\|\phi_{y}\|^{\frac{1}{2}}\|\phi_{y}\|ds+C\delta^{2}\eta_{0}^{\frac{1}{2}}\int_{0}^{\tau}
\|\psi_{y}\|^{2}ds+C\delta^{2}\eta_{0}^{\frac{1}{2}}\int_{0}^{\tau}\|\phi_{y}\|^{2}ds\\[5mm]
&\displaystyle+C\int_{0}^{\tau}\|\psi_{y}\|^{\frac{1}{2}}\|\psi_{yy}\|_{L^{2}}^{\frac{1}{2}}\|\phi_{y}\|^{\frac{1}{2}}
\|\phi_{y}\|^{\frac{1}{2}}\|\phi_{y}\|ds\\[5mm]
&\displaystyle \leq C\eta_{0}\int_{0}^{\tau}\|\psi_{y}\|\|\phi_{y}\|+\|\psi_{yy}\|\|\phi_{y}\|ds+
C\delta^{2}\eta_{0}^{\frac{1}{2}}\int_{0}^{\tau}\|\psi_{y}\|^{2}ds+
C\delta^{2}\eta_{0}^{\frac{1}{2}}\int_{0}^{\tau}\|\phi_{y}\|^{2}ds\\[5mm]
&\displaystyle+C\eta_{0}^{\frac{1}{2}}\int_{0}^{\tau}\|\psi_{y}\|\|\phi_{y}\|+\|\psi_{yy}\|\|\phi_{y}\|ds\\[5mm]
&\displaystyle \leq C\eta_{0}^{\frac{1}{2}}\int_{0}^{\tau}\|\psi_{y}\|^{2}+\|\phi_{y}\|^{2}+\|\psi_{yy}\|^{2}ds,
\end{array}
\end{equation}
thus, we have
\begin{equation}\label{3tttr}
\begin{array}{ll}
\displaystyle I_{13}\leq C\eta_{0}^{\frac{1}{2}}\int_{0}^{\tau}\|\psi_{y}\|^{2}+\|\phi_{y}\|^{2}+\|\psi_{yy}\|^{2}ds.
\end{array}
\end{equation}
The estimates of $I_{14}$ is
\begin{equation}\label{3ttr}
\begin{array}{ll}
\displaystyle I_{14}&\displaystyle=-\int_{0}^{\tau}\int(\tilde{U}_{y}(\dfrac{1}{v^{1+\alpha}}-\dfrac{1}{\tilde{V}^{1+\alpha}}))_{y}\phi_{y}dyds
\displaystyle+\int_{0}^{\tau}\int(p'(v)-p'(\tilde{V}))v_{y}\phi_{y}dyds\\[5mm]
&\displaystyle=-\int_{0}^{\tau}\int\tilde{U}_{yy}(\dfrac{1}{v^{1+\alpha}}-\dfrac{1}{\tilde{V}^{1+\alpha}})\phi_{y}dyds-\int_{0}^{\tau}\int\tilde{U}_{y}
(\dfrac{1}{v^{1+\alpha}}-\dfrac{1}{\tilde{V}^{1+\alpha}})_{y}\phi_{y}dyds\\[5mm]
&\displaystyle+\int_{0}^{\tau}\int p''(\eta)\phi(\phi_{y}+\tilde{V}_{y})\phi_{y}dyds\\[5mm]
&\displaystyle\leq C\delta^{2}\int_{0}^{\tau}\int e^{-c\delta|y-s\tau|}|\phi\phi_{y}|dyds+C\delta^{2}\int_{0}^{\tau}\int e^{-c\delta|y-s\tau|}
\big(|\phi\phi_{y}^{2}|+|\tilde{V}_{y}\phi\phi_{y}|+\phi_{y}^{2}\big)dyds\\[5mm]
&\displaystyle+C\int_{0}^{\tau}\int|\phi\phi_{y}^{2}|dyds+C\delta^{2}\int_{0}^{\tau}\int e^{-c\delta|y-s\tau|}|\phi\phi_{y}|dyds\\[5mm]
&\displaystyle\leq C\delta^{2}\int_{0}^{\tau}\int\phi^{2}dyds+C(\delta^{2}+\eta_{0}^{\frac{1}{2}})\int_{0}^{\tau}\int\phi_{y}^{2}dyds.
\end{array}
\end{equation}

Taking  $I_{10}$-$I_{14}$ into \eqref{36fr} and using Lemma 4.2, we get that
\begin{equation}\label{3t-0tr}
\begin{array}{ll}
\displaystyle\int(\dfrac{\phi_{y}^{2}}{2\tilde{V}^{1+\alpha}}-\phi_{y}\psi)dy+\int_{0}^{\tau}\int|p'(\tilde{V})|\phi_{y}^{2}dyds\\[5mm]
\displaystyle\leq C(1+\delta^{2}+\eta_{0}^{\frac{1}{2}})\int_{0}^{\tau}\int\psi_{y}^{2}dyds
\displaystyle +C(\delta^{2}+\eta_{0}^{\frac{1}{2}})\int_{0}^{\tau}\int\phi_{y}^{2}dyds+C\eta_{0}^{\frac{1}{2}}\int_{0}^{\tau}\int\psi_{yy}^{2}dyds+
C\delta^{3}.
\end{array}
\end{equation}

Using $\epsilon$-Young inequality and \eqref{3r2=2f}, we have
\begin{equation}\label{3hhr2=2f}
\begin{array}{ll}
\displaystyle\int\phi_{y}\psi dy\displaystyle\leq\epsilon\|\phi_{y}\|^{2}+C\|\psi\|^{2}
\displaystyle\leq \epsilon\|\phi_{y}\|^{2}+C\delta^{\frac{1}{3}}+\epsilon\int_{0}^{\tau}\|\phi_{y}\|^{2}ds,
\end{array}
\end{equation}
Thus,
\begin{equation}\label{3t-110tr}
\begin{array}{ll}
\displaystyle\|\ \phi_{y}\|^{2}+\int_{0}^{\tau}\| \phi_{y}\|^{2}dyds\\[5mm]
\displaystyle\leq C\delta^{\frac{1}{3}}+C(1+\delta^{2}+\eta_{0}^{\frac{1}{2}})\int_{0}^{\tau}\int\psi_{y}^{2}dyds+
C\eta_{0}^{\frac{1}{2}}\int_{0}^{\tau}\int\psi_{yy}^{2}dyds.
\end{array}
\end{equation}
Combining  \eqref{3t-110tr} with the above estimates, we can obtain
\begin{equation}\label{31rf}
\begin{array}{ll}
\displaystyle\|\phi\|_{H^{1}}^{2}+\|\psi\|^{2}+\int_{0}^{\tau}\| \phi_{y}\|^{2}ds+\int_{0}^{\tau}\|\psi_{y}\|^{2}ds
\displaystyle\leq C\delta^{\frac{1}{3}}+C\eta_{0}^{\frac{1}{2}}\int_{0}^{\tau}\| \psi_{yy}\|^{2}ds.
\end{array}
\end{equation}
The proof of  Lemma 4.3 is completed.
\end{proof}

\begin{lemma}
If $\delta_{0}$ and $\eta_{0}$ are suitable small, for  $0\leq \tau\leq T$, it holds for $\delta\leq\delta_{0}$ that
\begin{equation}\label{30090.5}
\begin{array}{ll}
\displaystyle \|\psi_{y}\|^{2}+\int_{0}^{\tau}\|\psi_{yy}\|^{2}ds
\displaystyle \leq C\delta+C(1+\delta^{2})\int_{0}^{\tau}\|\phi_{y}\|^{2}ds+C\delta^{2}\int_{0}^{\tau}\|\psi_{y}\|^{2}ds.
\end{array}
\end{equation}
\end{lemma}

\begin{proof}
Multiplying $\eqref{3=-0-09}_{2}$ by $-\psi_{yy}$ and integrating the resulting equation over $[0, \tau]\times R$, we have
\begin{equation}\label{3kk.5}
\begin{array}{ll}
\displaystyle 0=\int_{0}^{\tau}\int\psi_{\tau}\psi_{yy}dyds+\int_{0}^{\tau}\int p'(\tilde{V})\phi_{y}\psi_{yy}dyds-
\int_{0}^{\tau}\int(\dfrac{\psi_{y}}{v^{1+\alpha}})_{y}\psi_{yy}dyds-\int_{0}^{\tau}\int Q_{1}\psi_{yy}dyds\\[5mm]
\displaystyle=-\frac{1}{2}\int\psi_{y}^{2}dy+\int_{0}^{\tau}\int p'(\tilde{V})\phi_{y}\psi_{yy}dyds
-\int_{0}^{\tau}\int\dfrac{1}{v^{1+\alpha}}\psi_{yy}^{2}dyds-\int_{0}^{\tau}\int(\dfrac{1}{v^{1+\alpha}})_{y}\psi_{y}\psi_{yy}dyds\\[5mm]
\displaystyle-\int_{0}^{\tau}\int Q_{1}\psi_{yy}dyds,
\end{array}
\end{equation}
i.e.
\begin{equation}\label{3kek.5}
\begin{array}{ll}
\displaystyle\frac{1}{2}\int\psi_{y}^{2}dy+\int_{0}^{\tau}\int\dfrac{1}{v^{1+\alpha}}\psi_{yy}^{2}dyds\\[5mm]
\displaystyle=\int_{0}^{\tau}\int p'(\tilde{V})\phi_{y}\psi_{yy}dyds-\int_{0}^{\tau}\int(\dfrac{1}{v^{1+\alpha}})_{y}\psi_{y}\psi_{yy}dyds-\int_{0}^{\tau}\int Q_{1}\psi_{yy}dyds.
\end{array}
\end{equation}

Applying $\epsilon$-Young inequality and Proposition 2.1, we get
\begin{equation}\label{3wkek.5}
\begin{array}{ll}
\displaystyle\int_{0}^{\tau}\int p'(\tilde{V})\phi_{y}\psi_{yy}dyds
\leq\epsilon\int_{0}^{\tau}\int\psi_{yy}^{2}dyds+C\int_{0}^{\tau}\int\phi_{y}^{2}dyds.
\end{array}
\end{equation}

\begin{equation}\label{31kek.5}
\begin{array}{ll}
\displaystyle\int_{0}^{\tau}\int(\dfrac{1}{v^{1+\alpha}})_{y}\psi_{y}\psi_{yy}dyds\\[5mm]
\displaystyle\leq C\int_{0}^{\tau}\int|(\phi_{y}+\tilde{V}_{y})\psi_{y}\psi_{yy}|dyds\\[5mm]
\displaystyle\leq C\int_{0}^{\tau}\int|\phi_{y}\psi_{y}\psi_{yy}|dyds+C\int_{0}^{\tau}\int|\tilde{V}_{y}\psi_{y}\psi_{yy}|dyds\\[5mm]
\displaystyle\leq C\int_{0}^{\tau}\|\psi_{y}\|_{L^{\infty}}\int|\phi_{y}\psi_{yy}|dyds+C\delta^{2}\int_{0}^{\tau}\int|\psi_{y}\psi_{yy}|dyds\\[5mm]
\displaystyle\leq C\int_{0}^{\tau}\|\psi_{y}\|^{\frac{1}{2}}\|\psi_{yy}\|^{\frac{1}{2}}\|\phi_{y}\|\|\psi_{yy}\|dyds+
C\delta^{2}\int_{0}^{\tau}\int|\psi_{y}\psi_{yy}|dyds\\[5mm]
\displaystyle\leq C{\eta_{0}}^{\frac{1}{2}}\int_{0}^{\tau}\|\psi_{y}\|\|\psi_{yy}\|dyds+
C{\eta_{0}}^{\frac{1}{2}}\int_{0}^{\tau}\|\psi_{yy}\|^{2}dyds+C\delta^{2}\int_{0}^{\tau}\|\psi_{y}\|^{2}dyds\\[5mm]
\displaystyle+C\delta^{2}\int_{0}^{\tau}\|\psi_{yy}\|^{2}dyds\\[5mm]
\displaystyle\leq C({\eta_{0}}^{\frac{1}{2}}+\delta^{2})\int_{0}^{\tau}\|\psi_{y}\|^{2}dyds+
C({\eta_{0}}^{\frac{1}{2}}+\delta^{2})\int_{0}^{\tau}\|\psi_{yy}\|^{2}dyds,
\end{array}
\end{equation}
and the last term of \eqref{3kek.5} is estimated as follow,

\begin{equation}\label{3w-kek.5}
\begin{array}{ll}
\displaystyle\int_{0}^{\tau}\int  Q_{1}\psi_{yy}dyds\\[5mm]
\displaystyle=\int_{0}^{\tau}\int(\tilde{U}_{y}(\dfrac{1}{v^{1+\alpha}}-\dfrac{1}{\tilde{V}^{1+\alpha}}))_{y}\psi_{yy}dyds
-\int_{0}^{\tau}\int(p'(v)-p'(\tilde{V}))v_{y}\psi_{yy}dyds\\[5mm]
\displaystyle=-\int_{0}^{\tau}\int\tilde{U}_{yy}(\dfrac{1}{v^{1+\alpha}}-\dfrac{1}{\tilde{V}^{1+\alpha}})\psi_{yy}dyds-\int_{0}^{\tau}\int\tilde{U}_{y}
(\dfrac{1}{v^{1+\alpha}}-\dfrac{1}{\tilde{V}^{1+\alpha}})_{y}\psi_{yy}dyds\\[5mm]
\displaystyle-\int_{0}^{\tau}\int(p'(v)-p'(\tilde{V}))v_{y}\psi_{yy}dyds\\[5mm]
\displaystyle\leq C\delta^{2}\int_{0}^{\tau}\int e^{-c\delta|y-s\tau|}|\phi\psi_{yy}|dyds+C\delta^{2}\int_{0}^{\tau}\int e^{-c\delta|y-s\tau|}
\big(|\phi\phi_{y}\psi_{yy}|+|\tilde{V}_{y}\phi\psi_{yy}|\\[5mm]+\phi_{y}\psi_{yy}\big)dyds
\displaystyle+C\int_{0}^{\tau}\int|\phi\phi_{y}\psi_{yy}|dyds+C\delta^{2}\int_{0}^{\tau}\int e^{-c\delta|y-s\tau|}|\phi\psi_{yy}|dyds\\[5mm]
\displaystyle\leq C\delta^{2}\int_{0}^{\tau}\int(\phi^{2}+\phi_{y}^{2}+\psi_{yy}^{2})dyds\\[5mm]
\displaystyle\leq C\delta+ C\delta^{2}\int_{0}^{\tau}\int(\phi_{y}^{2}+\psi_{yy}^{2})dyds+
C\eta_{0}^{\frac{1}{2}}\int_{0}^{\tau}\int\psi_{y}^{2}dyds.
\end{array}
\end{equation}
Consequently, due to the smallness of $\epsilon, \delta$, we can obtain \eqref{30090.5}. The proof of Lemma 4.4 is completed.
\end{proof}

$\mathbf{Proof \ of \ Proposition \ 3.2.}$ Taking $\eta_{0}=\delta^{\frac{1}{4}}$, applying the smallness of $\delta$ and Lemma 4.1-Lemma 4.4, we can close a priori assumption \eqref{387.125} and then complete the proof of Proposition 3.2.

$\mathbf{Proof \ of \ Theorem \ 1.1. }$ Theorems 3.1 gives that there exists a sequence of smooth solutions $(v^{\alpha}, u^{\alpha})(t)$ to the Navier-Stokes system \eqref{11.1} with the well-prepared initial data \eqref{38.5} for all $t>0$. From \eqref{3ii9.5}, we have
\begin{equation}\label{3900.5}
\int_{0}^{\infty}\|(\Phi_{y},\Psi_{y})(\tau)\|^{2}d\tau+\int_{0}^{\infty}\big|\frac{d}{d\tau}\|(\Phi_{y},\Psi_{y})(\tau)\|^{2}\big|d\tau
<\infty,
\end{equation}
which implies  that there exists a sequence $\{\tau_{n}\}, n=1,2,3,...$ with
$\{\tau_{n}\}\rightarrow+\infty$ as $n\rightarrow+\infty$ such that
\begin{equation}\label{39tr00.5}
\|(\Phi_{y},\Psi_{y})(\tau_{n})\|^{2}\rightarrow0,\ \ \ \ as  \ \  n\rightarrow+\infty,
\end{equation}
and
$\lim\limits_{\tau\rightarrow+\infty}\|(\Phi_{y},\Psi_{y})(\tau)\|$ exists.  Consequently,
\begin{equation}\label{3or00.5}
\|(\Phi_{y},\Psi_{y})(\tau)\|^{2}\rightarrow 0, \ \ \ \ as\ \ \tau\rightarrow+\infty.
\end{equation}

Applying \eqref{3ii9.5}, \eqref{3or00.5} and Sobolev's interpolation inequality lead to
\begin{equation}\label{3=98ee090.5}
\begin{array}{ll}
\sup\limits_{y\in R}|(\Phi_{y},\Psi_{y})(y,\tau)|&\leq C\|(\Phi_{y},\Psi_{y})(\tau)\|^{\frac{1}{2}}\|(\Phi_{yy},\Psi_{yy})(\tau)\|^{\frac{1}{2}}\\[5mm]
&\leq C\|(\Phi_{y},\Psi_{y})(\tau)\|^{\frac{1}{2}}\rightarrow0,\ \ \ as\ \ \tau\rightarrow+\infty.
\end{array}
\end{equation}

Therefore, we can easily obtain the convergence \eqref{v3} from \eqref{3=98ee090.5} and Proposition 2.1. The proof of Theorem 1.1 is completed.

\end{document}